\documentclass[12pt]{amsart}
\usepackage{amsmath,amsfonts, amsthm, amssymb}
\usepackage{microtype}
\usepackage{graphicx}
\usepackage{tikz}
\usepackage{thm-restate}
\usepackage{tikz}
\usepackage{float}
\usetikzlibrary{arrows.meta,positioning}

\usepackage[backend=biber]{biblatex}
\def\lsim{{\,\lesssim\,}}

\newcommand{\diam}{\operatorname*{diam}}

\usepackage{xcolor}

\newtheorem{theorem}{Theorem}

\newtheorem{lemma}[theorem]{Lemma}
\newtheorem{remark}{Remark}

\title{Exponential Localization of Spatial Random Permutations in One Dimension}
\author{ Reuben Drogin, Felipe Hern{\'a}ndez}
\date{\today}

\begin{document}
\begin{abstract}
We consider a class of random permutations
of the interval $[-n,n]$, in which points are typically displaced
a distance $O(W)$. We show the cycles are localized 
on the scale $W^3$, with an exponentially decaying tail bound.
Analogous to eigenfunctions of one dimensional random band matrices, 
the cycles are conjectured to be localized to the scale $W^2.$
\end{abstract}
\maketitle

\section{Introduction}

For $p\in [1,\infty]$ and $W\geq 1$,
consider a permutation $\pi:[-n,n]\to[-n,n]$
drawn from the following Boltzmann-Gibbs distribution
\begin{equation}\label{def:gibbs-distribution}
\mathbb{P}_{p,W,n}(\pi)\propto \exp\left(-\frac{1}{W^p}\sum_{i\in [-n,n]}|\pi(i)-i|^p\right).
\end{equation}
If $p=\infty$ we take $\mathbb{P}_{\infty,W,n}$
to be the uniform distribution on permutations which displace points
at most $W$, given by
\begin{equation}\label{def:p-infty-dist}
\mathbb{P}_{p,W,n}(\pi)\propto 1_{\max\limits_{i\in [-n,n]} |\pi(i)-i|\leq W}.
\end{equation}
This model and related models are sometimes called ``spatial random permutations.''
See below for further background and related work. The parameter $W$ can be thought of as the typical displacement of a point,
and we prove the cycles of such permutations
are \textit{exponentially localized} to scale $W^3$.

\begin{theorem}\label{thm:main-boltzman}
For any $p\in [1,\infty]$, $\lambda, n,W\geq 1$ and $j\in [-n,n]$ we have
$$\mathbb{P}_{p,W,n}\left(\diam C_{\pi}(j)\geq \lambda\right)\leq 2\exp\left(-c\frac{\lambda}{W^3}\right),$$
where $c>0$ is uniform in all parameters, and $C_{\pi}(j)$ is the cycle of $j$ in $\pi$ given by
$$C_{\pi}(j):=\{\pi^k(j): k\in \mathbb{Z}\}.$$
\end{theorem}

\subsection{Background and Related Work}
The specific class of models considered in this paper has been considered before in various
works \cite{BR-2015Gibbs,GRU-2007,Muk-2016,B-2014}, and we mention two of particular relevance.
In \cite{FM-2021}, Fyodorov and Muirhead
conjectured that
$$\mathbb{E} \diam C_{\pi}(j)\sim \min(W^2, n),$$
and showed that when $p=1$ the permutations
exhibit ``band structure'' in the sense that $\mathbb{E}|\pi(0)-0| = c(1+o(1))W.$
Most recently, Zhong \cite{Zhong-2021}
proved this conjecture when $p=1$ using ``hit and run'' sampling.

The model \eqref{def:gibbs-distribution}
is part of a larger class of ``spatial random permutations'',
which have been considered as far back as the 1950s
by Feynman as a model for the quantum Bose gas in \cite{Feynman-1953} and
by Mallows as a model for ranking with noise in \cite{Mallows-1957}.
This class of models includes, among others, Manhattan pinball, the Lorentz mirror model, the interchange,
and the Mallows model in which one replaces the displacement in \eqref{def:gibbs-distribution}
with the number of inversions of $\pi$. The cycles in
all of these models are expected to exhibit
similar localization lengths, with this being
proven in dimension one in the Mallows model by Gladkich and Peled \cite{GAD-2012},
and in the interchange by Kozma and Sidoravicius in unpublished work.
The argument presented here is rather robust,
and can be adapted to those models, however,
without an estimate like the one in Remark \ref{rem:sharp},
it is not sharp by a power of $W$.

A major open question for this class of models, see \cite{T-93,BU-2009,GAD-2012,B-2014},
is to prove that cycles in these models are \textit{delocalized}, i.e. typically 
have size of order $n$,  if $d\geq 3$  
and localized, at the scale $O(e^{cW^2})$, if $d=2$, where $W$ is defined
properly for the model. Recently there has been
tremendous progress on this when $d\geq 5$ in various models
like the interchange, and Manhattan pinball,
see \cite{ES-2022,EAH-2025,EK-2026}. This is also known for
the interchange on the tree \cite{A-2003,H-2013}. Proving this in $d=2,3$ seems out of reach of current methods
(although see \cite{EP-2019} for a continuous model where this was shown).

Finally, we briefly mention the connection to a family of quantum models.
If one views the permutation $\pi$
as a matrix, then one can check that the entries of the eigenvectors
of this matrix have support on the cycles of $\pi$.
Hence, Theorem \ref{thm:main-boltzman} can be interpreted
as showing the eigenfunctions of this permutation matrix
are localized to scale $W^3$. The conjectured localization length
of $W^2$ is the same as what was conjectured
for eigenfunctions of \textit{random band matrices}. This conjecture
for one dimensional random band matrices was recently resolved, see \cite{YY-2025,ER-2025,D-2025}.
In $d\geq 2$ the eigenfunctions
of random band matrices are expected to exhibit the same
localization lengths as conjectured for cycles in spatial random permutations above.
For recent progress on this we refer the reader to \cite{DYYY-2025-2,DYYY-2025-3}.

\section{Proof of Theorem \ref{thm:main-boltzman}}
\subsection{Proof of Theorem 1 when $p=\infty$}
The proof is clearest when $p=\infty$ and we do that case first.
We prove it for $C_{\pi}(0)$, as the argument is the same for $C_{\pi}(j)$.
The idea is that if $C_{\pi}(0)$ crosses $\lambda + 2W>0$,
it must also cross $\lambda + 2W$ on the way back to $0$, see Figure 1 below.
Hence, one can find a transposition of two points $(a\ b)$ for which
$\max C_{\pi\circ (a\ b)}(0)\in (\lambda, \lambda + 2W].$
This allows us to bound the number of $\pi$ for
which $\max C_{\pi}(0)> \lambda + 2W$ by the
number of $\pi$ for which $\max C_{\pi}(0)\in (\lambda, \lambda + 2W].$
This argument is made precise below using the
``uncrossing map'' $\Phi_{\lambda}$.

\begin{lemma}\label{lem:one-step-infty}
For any $W,n\geq 1$ and $\lambda \geq 0$ we have
$$\mathbb{P}_{\infty, W,n}\left(\max C_{\pi}(0)> \lambda + 2W \right)\lsim W^2\mathbb{P}_{\infty, W,n}\left( \max C_{\pi}(0)\in (\lambda, \lambda + 2W]\right).$$
\end{lemma}

\begin{proof}
For any $\lambda \geq 0$ define the set $E_{\lambda} := \{\pi: \max C_{\pi}(0)> \lambda\}.$
It suffices to prove
\begin{equation}\label{eq:infty-turnaround}
\mathbb{P}_{\infty,W,n}(E_{\lambda+2W})\lsim W^2 \mathbb{P}_{\infty, W,n}(E_{\lambda+2W}^c \cap E_{\lambda})
\end{equation}
for all $\lambda \geq 0$. The idea is to define a map
$$\Phi_{\lambda}:E_{\lambda+2W}\cap S_{W} \to E_{\lambda+2W}^c\cap E_{\lambda}\cap S_{W},$$
where $S_W$ is the support of $\mathbb{P}_{\infty, W,n}$ given by
$$S_{W} := \{\pi: \max_{i\in [-n,n]} |\pi(i)-i|\leq W\},$$
such that
$$\max_{\tau\in E_{\lambda + 2W}^{c}\cap E_{\lambda}\cap S_W}\left|\Phi_{\lambda}^{-1}(\tau)\right|\lsim W^2.$$
The lemma follows since $\mathbb{P}:=\mathbb{P}_{\infty,W,n}$
is the uniform distribution on $S_{W}$ and so we can estimate
\begin{equation}\label{eq:pre-image-computation}
\begin{aligned}
\mathbb{P}(E_{\lambda + 2W})
&= \sum_{\tau \in E_{\lambda+2W}^{c}\cap E_{\lambda} \cap S_W}\mathbb{P}(\tau) \sum_{\pi\in \Phi_{\lambda}^{-1}(\tau)}\frac{\mathbb{P}(\pi)}{\mathbb{P}(\tau)}\\
& = \sum_{\tau \in E_{\lambda+2W}^{c}\cap E_{\lambda} \cap S_W}\mathbb{P}(\tau)\left|\Phi_{\lambda}^{-1}(\tau)\right|\\
& \lsim W^2\mathbb{P}(E_{\lambda+2W}^c\cap E_{\lambda}).
\end{aligned}
\end{equation}

To construct $\Phi_{\lambda}$ we take a $\pi\in E_{\lambda+2W}$
and apply a transposition which connects the first crossing of
$\lambda + 2W$ to the last crossing of $\lambda + 2W$. Precisely, for any
$\pi\in E_{\lambda+2W}\cap S_{W}$ define
\begin{align*}
i_{first}(\pi) & := \inf\{j\geq 0: \pi^{j}(0)\leq \lambda+2W < \pi^{j+1}(0)\},\\
i_{last}(\pi) & := \sup \{j\in[0,\text{Period}(\pi)) : \pi^{j+1}(0)\leq\lambda+2W<\pi^j(0)\},
\end{align*}
and
$$\Phi_{\lambda}(\pi) = \pi\circ (\pi^{i_{first}(\pi)+1}(0)\ \pi^{i_{last}(\pi)+1}(0)).$$
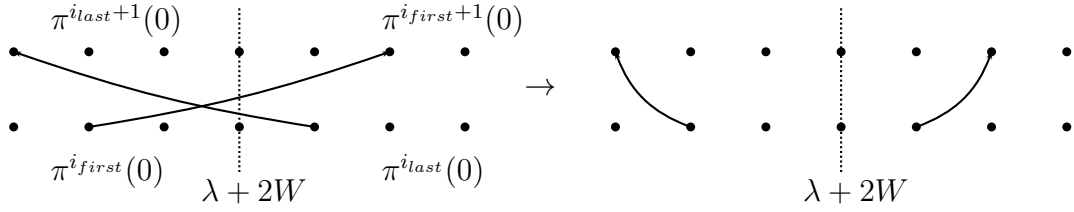
\begin{figure}[H]
\label{fig:perm-crossing}
  \centering

\begin{tikzpicture}[
  x=1.0cm, y=1.0cm,
  dot/.style={circle, fill, inner sep=1.2pt},
  >={Stealth[length=1.0mm]}
]

\path[use as bounding box] (-3.5,-0.2) rectangle (11.5,1.8);

\foreach \x in {-3,-2,-1,0,1,2,3} {
  \node[dot] at (\x,1) {};
  \node[dot] at (\x,0) {};
}


\node (topMid) at (0,1) {};
\node (botMid) at (0,0) {};

\draw[densely dotted, thick] (0,-0.6) -- (0,1.6);

\coordinate (bL) at (-2,0);
\coordinate (bR) at ( 1,0);
\coordinate (tL) at (-3,1);
\coordinate (tR) at ( 2,1);

\node[above right=2pt] at (-2.7,1) {$\pi^{i_{last}+1}(0)$};
\node[above right=2pt] at (1.7,-1) {$\pi^{i_{last}}(0)$};
\node[above right=2pt] at (1.7,1) {$\pi^{i_{first}+1}(0)$};
\node[above right=2pt] at (-2.7,-1) {$\pi^{i_{first}}(0)$};
\node[below right=2pt] at (-.7,-.5) {$\lambda+2W$};

\draw[->, thick, bend right=5]
  (bL) to node[midway, below, sloped] {} (tR);

\draw[->, thick, bend left=5]
  (bR) to node[midway, below, sloped] {} (tL);

\foreach \x in {5,6,7,8,9,10,11} {
  \node[dot] at (\x,1) {};
  \node[dot] at (\x,0) {};
}

\node at ( 4,.5) {$\to$};

\node (topMid) at (8,1) {};
\node (botMid) at (8,0) {};

\node[below right=2pt] at (7.3,-.5) {$\lambda+2W$};

\draw[densely dotted, thick] (8,-0.6) -- (8,1.6);

\coordinate (RbL) at (9,0);
\coordinate (RbR) at ( 6,0);
\coordinate (RtL) at (10,1);
\coordinate (RtR) at ( 5,1);


\draw[->, thick, bend right=25]
  (RbL) to node[midway, below, sloped] {} (RtL);

\draw[->, thick, bend left=25]
  (RbR) to node[midway, below, sloped] {} (RtR);

\end{tikzpicture}
  \caption{The action of the map $\Phi_{\lambda}$}
\end{figure}
Here we have used $(a\ b)$ for the permutation which transposes $a$ and $b$.
Note that $i_{first}(\pi)$ and $i_{last}(\pi)$ exist for any $\pi\in E_{\lambda+2W}$,
and furthermore for any
$\pi\in S_{W}\cap E_{\lambda + 2W}$ we have $\Phi_{\lambda}(\pi)\in E_{\lambda+2W}^c\cap E_{\lambda}\cap S_{W}$.

What is left is to establish that $|\Phi_{\lambda}^{-1}(\pi)|\lsim W^2$, but this
follows immediately since $\pi\in S_{W}$ implies that $\max_{i\in [-n,n]} |\pi(i)-i|\leq W$,
so the pair $(i_{last},i_{first})$ can take at most $\lsim W^2$ values.
The lemma follows from \eqref{eq:pre-image-computation}.
\end{proof}

The above lemma implies
$$\mathbb{P}_{\infty, W,n}\left(\max C_{\pi}(0)> \lambda + 2W \right)\leq (1-\frac{c}{W^2})\mathbb{P}_{\infty, W,n}( \max C_{\pi}(0)> \lambda),$$
and Theorem \ref{thm:main-boltzman} follows immediately by iterating this
and noting that the same argument can be applied to the minimum of a cycle.

\begin{remark}\label{rem:sharp}
The non-sharp $W^2$ above comes from
overestimating the size of the pre-image $\Phi_{\lambda}^{-1}(\pi)$.
To get the sharp factor of $W$ instead, it would
suffice to show that when $\max C_{\pi}(0)\leq \lambda$,
one typically has
$$\Phi^{-1}(\pi) = |\{(a,b):\pi\circ (a,b)\in S_{W},\ a\in C_{\pi}(0),\ and\ b,\pi(b)>\lambda\}|\lsim W.$$
Roughly, this should come from having $O(1)$ many choices for $a$
and $O(W)$ many choices for $b$ in the above set.
A heuristic explanation for this,
if $C_{\pi}(0)$ is viewed as a random walk,
is that random walks typically
only take $O(1)$ many steps near their maxima.
\end{remark}

\subsection{Proof of Theorem \ref{thm:main-boltzman} for $p\in [1,\infty)$.}
The proof in this case is nearly the same, except,
one must also account for how $\Phi_{\lambda}$
changes the energy of a permutation, given by $\sum_{i=-n}^n |\pi(i)-i|^p$.
Throughout the section we fix $p\in [1,\infty)$, $n,W\geq 1$, 
and let $\pi$ be sampled according to the distribution
in \eqref{def:gibbs-distribution}. All constants are uniform in $p,n, W$ and $\lambda$. 
The following lemma is analogous to Lemma \ref{lem:one-step-infty}, where the hard cutoff has now been replaced by an exponential weight.

\begin{lemma}\label{lem:p-turnaround}
For any $\lambda>0$ we have
\begin{align*}
\mathbb{P}(\max C_{\pi}(0)>\lambda)
  & \lsim W^2\sum_{j=0}^{\lambda} \exp\left(-\frac{|\lambda - j|^p}{W^p}\right)\mathbb{P}(\max C_{\pi}(0) = j).
\end{align*}
\end{lemma}

%
\begin{proof}
We again take $E_{\lambda}:= \{\pi: \max C_{\pi}(0)>\lambda \}$.
Fixing $\lambda> 0$ we take $i_{first}(\pi)$ and $i_{last}(\pi)$
the same as in the previous subsection but with $\lambda + 2W$ replaced 
by $\lambda$.
Note, they are well-defined for any $\pi\in E_{\lambda}$,
and thus we can also define
$$\Phi_{\lambda}:E_{\lambda }\to E_{\lambda}^{c},$$
via
$$\Phi_{\lambda}(\pi) = \pi \circ (\pi^{i_{first}(\pi)+1}(0)\ \pi^{i_{last}(\pi)+1}(0)).$$

Similar to \eqref{eq:pre-image-computation}, we have
$$\mathbb{P}(E_{\lambda}) = \sum_{\tau\in E_{\lambda}^c}\mathbb{P}(\tau)\sum_{\pi\in \Phi_{\lambda}^{-1}(\tau)}\frac{\mathbb{P}(\pi)}{\mathbb{P}(\tau)}.$$
Thus it suffices to show that for any $\tau \in E_{\lambda}^c$,
\begin{equation}\label{eq:suff-non-infty}
\sum_{\pi\in \Phi_{\lambda}^{-1}}\frac{\mathbb{P}(\pi)}{\mathbb{P}(\tau)}\lsim W^2 \exp\left(-\frac{|\lambda -\max C_{\tau}(0)|^p}{W^p}\right).
\end{equation}

Note that, as opposed to the $p=\infty$ case, $\mathbb{P}(\pi)$
depends on the energy of $\pi$, given by $\sum_{i} |\pi(i)-i|^p$.
So we must consider the \textit{energy}
as well as the \textit{entropy}.
The key observation is that $\Phi_{\lambda}$
\textit{uncrosses} points, see Figure 1, and thus decreases
the energy. Indeed, if $a, \tau (a)\leq\lambda < b, \tau(b)$,
then we have
\begin{align*}
\frac{\mathbb{P}(\tau \circ (\tau(a)\ \tau(b)))}{\mathbb{P}(\tau)}
& \leq \exp\left(-\frac{| \min(b, \tau(b))-\max(a, \tau(a))|^p}{W^p}\right).
\end{align*}
To apply this, note that any $\pi\in \Phi_{\lambda}^{-1}(\tau)$
equals $\tau\circ (\tau(a)\ \tau(b))$, for some $a$ and $b$ satisfying
the conditions above. Further, we must have
$a, \tau(a)\leq \max C_{\tau}(0)$, since $a$ is in the same cycle as $0$,
and $\tau(b),b>\lambda $.
Thus, summing over all such possibilities for $a$ and $b$ gives the estimate
\begin{align*}
\sum_{\pi\in\Phi_{\lambda}^{-1}(\tau)}\frac{\mathbb{P}(\pi)}{\mathbb{P}(\tau)}
& \lsim \sum_{a\leq\max C_{\tau}(0)}\sum_{b\geq\lambda } \exp\left(-\frac{|b-a|^p}{W^p}\right)\\
& \lsim \sum_{a\leq C_{\tau}(0)}W\exp\left(-\frac{|\lambda -a|^p}{W^p}\right)\\
& \lsim W^2\exp\left(-\frac{|\lambda -\max C_{\tau}(0)|^p}{W^p}\right),
\end{align*}
which implies \eqref{eq:suff-non-infty} and thus the lemma.
\end{proof}

Now we complete the proof of Theorem~\ref{thm:main-boltzman}.  
Set $p_j := \mathbb{P}(\max C_{\pi}(0)> j)$, so that Lemma~\ref{lem:p-turnaround}
implies there exists $C_0>0$ such that for any $k \geq 0$
$$p_{k+1}\leq C_0 W^2 \sum_{j=0}^{k}  \exp\left(-\frac{|k-j|^p}{W^p}\right) (p_{j}-p_{j+1}).$$
Setting $w_{j}^{(k)} := \exp\left(-|k-j|^p/W^p\right)$ for $0\leq j\leq k$
and  summing by parts yields
\[
p_{k+1} \leq C_0 W^2 \Big[- p_{k+1}  + \sum_{j=0}^k p_j (w_j^{(k)} - w_{j-1}^{(k)})\Big],
\]
with the convention $w_{-1}^{(k)} = 0$.  Rearranging one obtains
\[
p_{k+1} \leq (1-cW^{-2}) \sum_{j=0}^k p_j (w_j^{(k)}-w_{j-1}^{(k)})
\]
One can check the bound $p_{k}\leq 2\exp(-c_0\frac{k}{W^3})$ is propagated under this recurrence for $c_0>0$ small enough, 
proving the theorem.  
\printbibliography
\end{document}